\documentclass[a4paper,10pt,twoside]{article}

\usepackage{tensor}
\usepackage[utf8]{inputenc}
\usepackage[UKenglish]{babel}
\usepackage{cite}
\usepackage{amssymb}
\usepackage{amsmath}
\usepackage{mathrsfs}
\usepackage{amsthm}
\usepackage[T1]{fontenc}
\usepackage{accents}
\usepackage[cm]{fullpage}
\usepackage{graphicx}
\usepackage{tensor}
\usepackage{relsize}
\usepackage{appendix}
\usepackage[affil-it]{authblk}
\usepackage{tensor}
\usepackage{verbatim}
\usepackage{comment}

\usepackage{xcolor}

\allowdisplaybreaks

\usepackage{mathtools}
\mathtoolsset{showonlyrefs,showmanualtags}

\RequirePackage{lastpage}
\RequirePackage{latexsym}
\makeatletter
\ifx\@NODS\undefined\RequirePackage{dsfont}\fi
\ifx\@NOAMS\undefined\RequirePackage{amsmath,amsfonts,amssymb,amsthm}\fi
\makeatother

\RequirePackage{geometry}
\RequirePackage{bera}
\RequirePackage[pdftex,pagebackref=false]{hyperref}
\hypersetup{pdfborder=0 0 0}
\hypersetup{pdfstartview={FitH}}

\title{{\Large \bfseries{Derivatives of Feynman-Kac semigroups}}}

\author{ }
\author{James Thompson%
  \footnote{University of Luxembourg, Email: \texttt{james.thompson@uni.lu}}}

\date{\today}

\makeatletter
\def\@MRExtract#1 #2!{#1}
\newcommand{\MR}[1]{
  \xdef\@MRSTRIP{\@MRExtract#1 !}%
  \href{http://www.ams.org/mathscinet-getitem?mr=\@MRSTRIP}{MR-\@MRSTRIP}}
\makeatother

\makeatletter
\renewenvironment{thebibliography}[1]{%
  \section*{\refname
    \@mkboth{\MakeUppercase\refname}{\MakeUppercase\refname}}%
  \phantomsection%
  \addcontentsline{toc}{section}{\refname}%
  \list{\@biblabel{\@arabic\c@enumiv}}%
  {\settowidth\labelwidth{\@biblabel{#1}}%
    \small%
    \setlength{\labelsep}{0.4em}%
    \setlength{\leftmargin}{\labelwidth}%
    \addtolength{\leftmargin}{\labelsep}%
    \setlength{\itemsep}{-.25em}%
    \@openbib@code
    \usecounter{enumiv}%
    \let\p@enumiv\@empty
    \renewcommand\theenumiv{\@arabic\c@enumiv}}%
  \sloppy\clubpenalty4000\@clubpenalty\clubpenalty\widowpenalty4000%
  \sfcode`\.\@m}{%
  \def\@noitemerr{%
    \@latex@warning{Empty `thebibliography' environment}}%
  \endlist}
\makeatother

\makeatletter
\ifx\@NODS\undefined%

\let\mathbb=\mathds
\else%
\fi
\makeatother

\DeclareMathOperator{\divv}{div}

\DeclareMathOperator{\tr}{tr}
\DeclareMathOperator{\inj}{inj}
\DeclareMathOperator{\Aut}{Aut}

\def\<{\langle}
\def\>{\rangle}

\def\Ric{\mathop{\rm Ric}}
\def\id{\mathop{\rm id}}
\def\V{\mathbb V}

\newcommand{\E}{\mathbb{E}}

\newtheorem{theorem}{Theorem}[section]
\newtheorem{lemma}[theorem]{Lemma}

\newtheorem{corollary}[theorem]{Corollary}

\newtheorem{remark}[theorem]{Remark}

\begin{document}

\maketitle

\begin{abstract}%
   \noindent%
   {We prove Bismut-type formulae for the first and second derivatives of a Feynman-Kac semigroup on a complete Riemannian manifold. We derive local estimates and give bounds on the logarithmic derivatives of the integral kernel. Stationary solutions are also considered. The arguments are based on local martingales, although the assumptions are purely geometric.}\\[1em]%
   {\footnotesize%
     \par%
     \par
   }
\end{abstract}

\section{Introduction}

Suppose $M$ is a complete Riemannian manifold of dimension $n$ with Levi-Civita connection $\nabla$. Denote by $\Delta$ the Laplace-Beltrami operator, suppose $Z$ is a smooth vector field and set $L:=\frac{1}{2}\Delta +Z$. Any elliptic diffusion operator on a smooth manifold induces, via its principle symbol, a Riemannian metric with respect to which it takes this form. Denote by $x_t$ a diffusion on $M$ starting at $x_0 \in M$ with generator $L$ and explosion time $\zeta(x_0)$. The explosion time is the random time at which the process leaves all compact subsets of $M$. Suppose $V:[0,\infty) \times M \rightarrow \mathbb{R}$ is a smooth function which is bounded below and denote by $P^V_t f$ the associated Feynman-Kac semigroup, acting on bounded measurable functions $f$. For $T>0$ fixed, $P^V_tf$ is smooth and bounded on $(0,T] \times M$, satisfies the parabolic equation
\begin{equation}\label{eq:shrod}
\partial_t \phi_t = (L - V_t)\phi_t
\end{equation}
on $(0,T] \times M$ with $\phi_0 = f$ and for
\begin{equation}\label{eq:defnVt}
\V_t := e^{-\int_0^t V_{T-s}(x_s) ds}
\end{equation} 
is represented probabilistically by the Feynman-Kac formula
\begin{equation}\label{eq:fkform}
P^V_Tf(x_0) = \E \left[ \V_T f(x_T)\mathbf{1}_{\lbrace T < \zeta(x_0)\rbrace} \right].
\end{equation}
In the self-adjoint case, equation \eqref{eq:shrod} corresponds (via Wick rotation) to the Schrödinger equation for a single non-relativistic particle moving in an electric field in curved space. In this sense, the derivative $d P^V_Tf$ corresponds to the momentum of the particle and $L P^V_Tf$ the kinetic energy.

In this article, we prove probabilistic formulae and estimates for $d P^V_Tf$, $L P^V_Tf$ and $\nabla d P^V_Tf$.  In doing so, we extend results in \cite{Thalmaier97} (by including $V$) and in \cite{APT} (by including $Z$ and $V$). In each case, we allow for unbounded and time-dependent $V$. Our approach is more concise than that of \cite{APT}, since we avoid the extrinsic argument in favour of the differential Bianchi identity. Our results imply new Bismut-type formulae for the derivatives of the heat kernel in the \textit{forward} variable (see, for example, Corollary \ref{cor:forbis}).

Our formula for $dP^V_Tf$ is given by Theorem \ref{thm:locformone}. For $v \in T_{x_0}M$ it states
\begin{equation}
(dP^V_Tf)(v)  = - \mathbb{E}\left[ \V_T f(x_T)\mathbf{1}_{\lbrace T < \zeta(x_0)\rbrace} \int_0^{T} \langle W_s (\dot{k}_s v),//_sdB_s\rangle+dV_{T-s}(W_s (k_sv)) ds \right]
\end{equation}
where $//_t$ and $W_t$ are the usual parallel and damped parallel transports, respectively, and $B_t$ the martingale part of the antidevelopment of $x_t$ to $T_{x_0}M$. The process $k_t$ is chosen so that it vanishes once $x_t$ exits a regular domain (an open connected subset with compact closure and smooth boundary).  Imposing this condition on $k_t$ obviates the need for any assumptions on ${\Ric}_Z$. Conversely, if we assume ${\Ric}_Z$ is bounded below then we can choose $k_t = (T-t)/T$ and our formula for $dP^V_tf$ reduces to that of \cite[Theorem~5.2]{ElworthyLi}. Formulae in \cite{ElworthyLi} are derived from the assumption that one can differentiate under the expectation, and thus require global assumptions. Our approach, on the other hand, follows that of \cite{Thalmaier97} and \cite{APT} in using local martingales to obtain local formula for which no assumptions are needed.

Our formula for $LP^V_Tf$ is given by Theorem \ref{thm:formhalf}. It states
\begin{equation}
\begin{split}
&L(P^V_Tf)(x_0) \\
=\text{ }&\E\left[ \V_T f(x_T)\mathbf{1}_{\lbrace T<\zeta(x_0)\rbrace}\int_0^T \langle \dot{k}_s Z,//_s dB_s\rangle \right]\\
& +\frac{1}{2}\E\left[ \V_T f(x_T)\mathbf{1}_{\lbrace T<\zeta(x_0)\rbrace}\left( \int_0^T \<  W_s \dot{l}_s ,//_s dB_s\> + dV_{T-s}(W_s l_s)ds\right)\int_0^T \dot{k}_s W_s^{-1} //_s dB_s\right]
\end{split}
\end{equation}
where the processes $k$ and $l$ are assumed to vanish outside of a regular domain. A formula for $\Delta P_T $ (acting on differential forms) was previously given in \cite{EL3}, for the case of a compact manifold with $Z=0$ and $V=0$.

Our formula for $\nabla dP_tf$ is given by Theorem \ref{thm:locformtwo}. For $v,w \in T_{x_0}M$ it states
\begin{equation}
\begin{split}
&(\nabla d P^V_Tf)(v,w) \\
=\text{ }&  - \E \left[ \V_T f(x_T)\mathbf{1}_{\lbrace T < \zeta(x_0)\rbrace}\int_0^T\langle  W^{\prime}_s(\dot{k}_sv,w), //_s dB_s\rangle\right]\\
&-\E\left[ \V_T f(x_T)\mathbf{1}_{\lbrace T < \zeta(x_0)\rbrace}\int_0^T ((\nabla d V_{T-s})(W_s(k_sv),W_s(w))+(dV_{T-s})(W^{\prime}_s(k_sv,w)))ds\right]\\
&+\E\left[ \V_T f(x_T)\mathbf{1}_{\lbrace T < \zeta(x_0)\rbrace}\left(\int_{0}^T \< W_s( \dot{l}_sw),//_s dB_s\> + dV_{T-s}( W_s(l_sw)) ds\right)\right.\\
&\quad\quad\quad\quad\quad\quad\quad\quad\quad\quad\quad\quad\quad\quad\quad\left.\cdot\left(\int_0^{T} \<W_s (\dot{k}_sv),//_sdB_s\>+ dV_{T-s}(W_s(k_sv))ds\right)\right]
\end{split}
\end{equation}
where $W^{\prime}_t$ solves a covariant It\^{o} equation determined by the curvature tensor and its derivative. This extends \cite[Theorem~2.1]{APT} while avoiding the use of a stochastic differential equation.

The formulae mentioned above are derived in Section \ref{sec:localform}. Solutions to the time independent equation
\begin{equation}
(L - V)\phi = -E \phi
\end{equation}
with $E \in \mathbb{R}$ are subject to a similar analysis, as outlined in Section \ref{sec:statsolns}. In Section \ref{sec:localest} we derive local estimates, using the formulae of Section \ref{sec:localform} and local assumptions on curvature and the derivative of the potential function. We do so by choosing the processes $k$ and $l$ appropriately, as in \cite{Thalmaier97} and \cite{APT}, and applying the Cauchy-Schwarz inequality. These local estimates are given by Theorems \ref{thm:locestone}, \ref{thm:locestonehalf} and \ref{thm:locesttwo}; global estimates are then given as corollaries. The global estimates are derived under appropriate global assumptions and imply the boundedness of $dP^V_tf$, $LP^V_tf$ and $\nabla dP^V_tf$ on $[\epsilon,T]\times M$. These bounds lead to the non-local formulae of Section \ref{sec:globform}, in which the processes $k$ and $l$ are chosen deterministically. For the case in which $Z$ is a gradient, estimates on the logarithmic derivatives of the integral kernel can then be derived, using Jensen's inequality. They are given in Section \ref{sec:kerests} and extend those of \cite{HsuEstimates} and \cite{StroockTuretsky}.

\section{Local Formulae}\label{sec:localform}

For the remainder of this article, we fix $T>0$ and set $f_t := P^V_{T-t}f$.

\subsection{Gradient}\label{ss:firstderlocform}

Denote by ${\Ric}^\sharp_Z := {\Ric}^\sharp - 2\nabla Z$ the Bakry-Emery tensor (see \cite{BL2}). Then the damped parallel transport $W_t: T_{x_0}M \rightarrow T_{x_t}M$ is the solution, along the paths of $x_t$, to the covariant ordinary differential equation
\begin{equation}\label{eq:eqfordpt}
DW_t = - \frac{1}{2}{\Ric}^\sharp_Z W_t
\end{equation}
with $W_0 = \id_{T_{x_0}M}$. Suppose $D$ is a regular domain in $M$ with $x_0 \in D$ and denote by $\tau$ the first exit time of $x_t$ from $D$.

\begin{lemma}\label{lem:locmartone}
Suppose $v \in T_{x_0}M$ and that $k$ is a bounded adapted process with paths in the Cameron-Martin space $L^{1,2}([0,T];\Aut(T_{x_0}M))$ such that $k_t = 0$ for $t \geq T-\epsilon$. Then
\begin{equation}\label{eq:locmartingone}
\V_t df_t (W_t (k_tv)) -\V_tf_t(x_t)\int_0^t \< W_s (\dot{k}_sv),//_s dB_s\> -\int_0^t \V_sf_s(x_s)dV_{T-s}(W_s (k_sv)) ds
\end{equation}
is a local martingale on $[0,\tau \wedge T)$.
\end{lemma}

\begin{proof}
Setting $N_t(v) := df_t(W_t(v))$ we see by It\^{o}'s formula and the relations
\begin{equation}
\begin{split}
d \Delta f =& \tr \nabla^2 df - df({\Ric}^\sharp)\\
dZf = & \nabla_Z df + df(\nabla Z)\\
dV_tf = & f dV_t + V_t df
\end{split}
\end{equation}
(the first one is the Weitzenb\"{o}ck formula) that
\begin{equation}
\begin{split}
dN_t(v) \stackrel{m}{=}\text{ }& df_t(DW_t(v))dt+(\partial_t df_t)(W_t(v))dt + \left(\frac{1}{2}\tr \nabla^2 + \nabla_Z\right)(df_t)(W_t(v))dt \\
=\text{ }& V_{T-t} N_t(v) dt +f_t(x_t) dV_{T-t}(W_t(v))dt
\end{split}
\end{equation}
where $\stackrel{m}{=}$ denotes equality modulo the differential of a local martingale. Recalling the definition of $\V_t$ given by equation \eqref{eq:defnVt}, it follows that
\begin{equation}
d(\V_t N_t(k_t v) ) \stackrel{m}{=} \V_t N_t(\dot{k}_t v) dt + \V_t f_t(x_t) dV_{T-t}(W_t( k_tv))dt
\end{equation}
so that
\begin{equation}
\V_t N_t( k_tv) -\int_0^t\V_s df_s(W_s(\dot{k}_sv))ds - \int_0^t \V_sf_s(x_s) dV_{T-s}(W_s(k_sv))ds 
\end{equation}
is a local martingale. By the formula
\begin{equation}
\V_t f_t(x_t) = f_0(x_0) + \int_0^t \V_s df_s(//_s dB_s)
\end{equation}
and integration by parts we see that
\begin{equation}
\int_0^t\V_s df_s( W_s(\dot{k}_sv))ds - \V_t f_t(x_t) \int_0^t \langle W_s (\dot{k}_sv),//_sdB_s\rangle
\end{equation}
is also a local martingale and so the lemma is proved.
\end{proof}

\begin{theorem}\label{thm:locformone}
Suppose $x_0\in D$ with $v \in T_{x_0}M$, $f \in \mathcal{B}_b$, $V$ bounded below and $T>0$. Suppose $k$ is a bounded adapted process with paths belonging to the Cameron-Martin space $L^{1,2}([0,T];\Aut(T_{x_0}M))$, such that $k_0=1$, $k_t=0$ for $t \geq \tau \wedge T$ and $\int_0^{\tau \wedge T} |\dot{k}_s|^2ds \in L^1$. Then
\begin{equation}\label{eq:formulalocone}
(dP^V_Tf)(v)  = - \mathbb{E}\left[ \V_T f(x_T)\mathbf{1}_{\lbrace T < \zeta(x_0)\rbrace} \int_0^{T} \langle W_s (\dot{k}_s v),//_sdB_s\rangle+dV_{T-s}(W_s (k_sv)) ds \right].
\end{equation}
\end{theorem}

\begin{proof}
As in the proof of \cite[Theorem~2.3]{Thalmaier97}, the process $k_t$ can be modified to $k^\epsilon_t$ so that $k^\epsilon_t = k_t$ for $t \leq \tau \wedge (T-2\epsilon)$ and $k^\epsilon_t = 0$ for $t \geq \tau \wedge (T-\epsilon)$, cutting off appropriately in between. Since $(df_t)_x$ is smooth and therefore bounded for $(t,x) \in [0,T-\epsilon]\times D$, it follows from Lemma \ref{lem:locmartone} and the strong Markov property that formula \eqref{eq:formulalocone} holds with $k^\epsilon_t$ in place of $k_t$. The result follows by taking $\epsilon \downarrow 0$.
\end{proof}

Denoting by $p^Z_T(x,y)$ the transition density of the diffusion with generator $L$, using Theorem \ref{thm:locformone} we can easily obtain the following Bismut formula, for the derivative of $p^Z_T(x,y)$ the in the \textit{forward} variable $y$.

\begin{corollary}\label{cor:forbis}
Suppose $x_0\in D$ with $\divv Z$ bounded below and $k$ as in Theorem \ref{thm:locformone}. Then
\begin{equation}
d\log p^Z_T(y,\cdot)_{x_0} = -\frac{\mathbb{E}\left[ e^{-\int_0^T \divv Z(x_s)ds} \int_0^{T} \langle W_s \dot{k}_s,//_sdB_s\rangle+d (\divv Z)(W_s k_s) ds \big\vert x_T = y\right]}{\mathbb{E}\left[ e^{-\int_0^T \divv Z(x_s)ds} \big\vert x_T = y\right]}
\end{equation}
where here $x_t$ is a diffusion on $M$ with generator $\frac{1}{2}\Delta -Z$ starting at $x_0$.
\end{corollary}

\begin{proof}
According to the Fokker-Planck equation, we have
\begin{equation}
p^Z_T(x,y) = p^{-Z,-\divv Z}_T(y,x)
\end{equation}
where $p^{-Z,-\divv Z}_T(y,x)$ denotes the minimal integral kernel for the semigroup generated by the operator $L^\ast = \frac{1}{2}\Delta -Z -\divv Z$. The result is therefore obtained simply by conditioning in Theorem \ref{thm:locformone}, having replaced $Z$ with $-Z$ and $V$ with $\divv Z$.
\end{proof}

\subsection{Generator}\label{ss:secderlocformhalf}

Now suppose $D_1$ and $D_2$ are regular domains with $x_0 \in D_1$ and $\overline{D_1}\subset D_2$. Denote by $\sigma$ and $\tau$ the first exit times of $x_t$ from $D_1$ and $D_2$, respectively.

\begin{lemma}\label{lem:locmartonehalf}
Suppose $x_0 \in D_1$ and $0<S<T$ and that $k,l$ are bounded adapted processes with paths in the Cameron-Martin space $L^{1,2}([0,T];\mathbb{R})$ such that $k_s=0$ for $s \geq \sigma \wedge S$, $l_s = 1$ for $s \leq \sigma \wedge S$ and $l_s = 0$ for $s \geq \tau \wedge (T-\epsilon)$. Then
\begin{equation}\label{eq:locmartingonehalf} 
\begin{split}
&\V_t(L f_t)(x_t)k_t -\frac{1}{2}\V_t df_t\left(W_t l_t\int_0^t \dot{k}_s W_s^{-1} //_s dB_s\right)\\
&\quad+\frac{1}{2}\V_t f_t(x_t)\int_0^t \< W_s \dot{l}_s,//_s dB_s\>\int_0^t \dot{k}_s W_s^{-1} //_s dB_s\\
&\quad\quad +\frac{1}{2}\int_0^t \V_s f_s(x_s) dV_{T-s}( W_sl_s) ds \int_0^t \dot{k}_s W_s^{-1}//_sdB_s\\
&\quad\quad\quad+\V_t f_t(x_t)\int_0^t \langle \dot{k}_s Z- k_s \nabla V_{T-s},//_s dB_s\rangle \\
&\quad\quad\quad\quad-\int_0^t \V_s f_s k_s LV_{T-s}ds
\end{split}
\end{equation}
is a local martingale on $[0,\tau \wedge T)$.
\end{lemma}

\begin{proof}
Defining
\begin{equation}
n_t := \left(L f_t \right)(x_t)
\end{equation}
we have, by It\^{o}'s formula, that
\begin{equation}
\begin{split}
d n_t =\text{ }& d(L f_t)_{x_t}//_t dB_t +\partial_t (L f_t)(x_t)dt + L(L f_t)(x_t) dt\\
=\text{ }& d(L f_t)_{x_t}//_t dB_t + L(V_{T-t}f_t)dt\\
=\text{ }&  d(L f_t)_{x_t}//_t dB_t + (L V_{T-t})f_t dt+ V_{T-t} n_t dt + \langle df_t, dV_{T-t}\rangle dt.
\end{split}
\end{equation}
It follows that
\begin{equation}
d( \V_t n_t k_t) \stackrel{m}{=} \V_t n_t \dot{k}_t  + k_t \V_t ( f_t L V_{T-t}+ \langle df_t, dV_{T-t}\rangle ) dt
\end{equation}
and so
\begin{equation}
\V_t(L f_t)(x_t)k_t - \int_0^t \V_s (L f_s)(x_s)\dot{k}_sds - \int_0^t \V_s k_s (f_s LV_{T-s} + \langle df_s, dV_{T-s}\rangle )ds
\end{equation}
is a local martingale, with
\begin{equation}
-(L f_t) (x_t)\dot{k}_tdt= \left(\frac{1}{2}d^\ast d-Z\right) f_t (x_t)\dot{k}_t dt= \left(\frac{1}{2}d^\ast (df_t)-(df_t)(Z)\right)\dot{k}_t dt.
\end{equation}
By the Weitzenbock formula
\begin{equation}
d((df_t)(W_t)) = (\nabla_{//_t dB_t}df_t)(W_t) -V_{T-t} (df_t)(W_t) dt +f_t(x_t) dV_{T-t}(W_t)dt
\end{equation}
from which it follows that
\begin{equation}
d(\V_t(df_t)(W_t)) = \V_t(\nabla_{//_t dB_t}df_t)(W_t) + \V_t f_t(x_t) dV_{T-t}(W_t)dt.
\end{equation}
Consequently, for an orthonormal basis $\lbrace e_i \rbrace_{i=1}^n$ of $T_{x_0}M$, by integration by parts we have
\begin{equation}\label{eqn:alocalmart}
\begin{split}
\V_t d^\ast (df_t) \dot{k}_tdt=&-\sum_{i=1}^n \V_t (\nabla_{//_t e_i} df_t)(//_t e_i) \dot{k}_tdt\\
\stackrel{m}{=}&-\V_t(\nabla_{//_t dB_t} df_t) (W_t\dot{k}_tW_t^{-1} //_t dB_t)\\
=&-d(\V_t df_t(W_t \int_0^t \dot{k}_sW_s^{-1} //_s dB_s))\\
&+d\left(\int_0^t \V_s f_s(x_s) dV_{T-s}(W_s)ds \int_0^t \dot{k}_sW_s^{-1} //_s dB_s\right).
\end{split}
\end{equation}
Furthermore
\begin{equation}
\int_0^t\V_s df_s( Z)\dot{k}_s ds - \V_t f_t(x_t) \int_0^t \langle \dot{k}_s Z  ,//_sdB_s\rangle
\end{equation}
is a local martingale and therefore
\begin{equation}
\begin{split}
&\int_0^t \V_s (L f_s)(x_s)\dot{k}_sds \\
&\quad -\frac{1}{2}\V_t df_t(W_t \int_0^t \dot{k}_sW_s^{-1} //_s dB_s)+\frac{1}{2} \int_0^t \V_s f_s(x_s) dV_{T-s}(W_s)ds \int_0^t \dot{k}_sW_s^{-1} //_s dB_s\\
&\quad\quad +\V_t f_t(x_t)\int_0^t \langle \dot{k}_s Z,//_s dB_s\rangle
\end{split}
\end{equation}
is also a local martingale. By the assumptions on $k$ and $l$ it follows from Lemma \ref{lem:locmartone} that
\begin{eqnarray}
O^1_t& =& \V_tdf_t(W_t((l_t-1))) -\mathbb{V}_tf_t(x_t)\int_0^t \< W_s( \dot{l}_s),//_s dB_s\> \\
& &\quad -\int_0^t \V_s f_s(x_s) dV_{T-s}( W_s((l_s-1))) ds,\\
O^2_t &= & \int_0^t \dot{k}_s W_s^{-1}//_sdB_s
\end{eqnarray}
are two local martingales. So the product $O^1_tO^2_t$ is also a local martingale, since $O^1 = 0$ on $[0,\sigma \wedge S]$ with $O^2$ constant on $[\sigma \wedge S,\tau \wedge (T-\epsilon))$. Consequently
\begin{equation}
\begin{split}
&- \V_t df_t\left(W_t\int_0^t \dot{k}_s W_s^{-1} //_s dB_s\right)+\V_t df_t\left(W_t l_t\int_0^t \dot{k}_s W_s^{-1} //_s dB_s\right)\\
&-\V_t f_t(x_t)\int_0^t \< W_s \dot{l}_s,//_s dB_s\>\int_0^t \dot{k}_s W_s^{-1} //_s dB_s\\
& -\int_0^t \V_s f_s(x_s) dV_{T-s}( W_s((l_s-1))) ds \int_0^t \dot{k}_s W_s^{-1}//_sdB_s
\end{split}
\end{equation}
is a local martingale and therefore so is
\begin{equation}
\begin{split}
&\V_t(L f_t)(x_t)k_t -\frac{1}{2}\V_t df_t\left(W_t l_t\int_0^t \dot{k}_s W_s^{-1} //_s dB_s\right)\\
&+\frac{1}{2}\V_t f_t(x_t)\int_0^t \< W_s \dot{l}_s,//_s dB_s\>\int_0^t \dot{k}_s W_s^{-1} //_s dB_s\\
& +\frac{1}{2}\int_0^t \V_s f_s(x_s) dV_{T-s}( W_sl_s) ds \int_0^t \dot{k}_s W_s^{-1}//_sdB_s\\
&+\V_t f_t(x_t)\int_0^t \langle \dot{k}_s Z,//_s dB_s\rangle\\
& - \int_0^t \V_s k_s (f_s LV_{T-s} + \langle df_s, dV_{T-s}\rangle )ds.
\end{split}
\end{equation}
Since
\begin{equation}
\int_0^t \V_s df_s(\nabla V_{T-s}) k_s ds - \V_t f_t(x_t) \int_0^t \langle k_s \nabla V_{T-s} ,//_sdB_s\rangle
\end{equation}
is a local martingale, the result follows.
\end{proof}

\begin{lemma}\label{lem:potlemma}
Suppose $x_0\in D_1$, $f \in \mathcal{B}_b$, $V$ bounded below and $0<S<T$. Suppose $k$ is a bounded adapted process with paths in the Cameron-Martin space $L^{1,2}([0,T];\mathbb{R})$ such that $k_s=0$ for $s \geq \sigma \wedge S$, $k_0=1$ and $\int_0^{\sigma \wedge S} |\dot{k}_s|^2ds \in L^1$. Then
\begin{equation}
V_T(x_0) P^V_Tf(x_0) = \E \left[ \V_T f(x_T)\mathbf{1}_{\lbrace T<\zeta(x_0)\rbrace} \int_0^T (k_s \dot{V}_{T-s}(x_s)-\dot{k}_s V_{T-s}(x_s))ds\right].
\end{equation}
\end{lemma}

\begin{proof}
By It\^{o}'s formula, we have
\begin{equation}
d(\V_t V_{T-t} f_t k_t) \stackrel{m}{=} -k_t \V_t \dot{V}_{T-t} f_t + \dot{k}_t \V_t V_{T-t} f_t
\end{equation}
which implies
\begin{equation}
\V_t V_{T-t}f_t k_t - \int_0^t ( \dot{k}_s \V_s V_{T-s} f_s -k_s \V_s \dot{V}_{T-s}f_s)ds
\end{equation}
is a local martingale on $[0,\tau \wedge T)$. The assumptions on $f$ and $V$ imply it is a martingale on $[0,\tau \wedge T]$, so result follows by taking expectations and applying the strong Markov property.
\end{proof}

\begin{theorem}\label{thm:formhalf}
Suppose $x_0 \in D_1$, $f \in \mathcal{B}_b$, $V$ bounded below and $0<S<T$. Suppose $k,l$ are bounded adapted processes with paths in the Cameron-Martin space $L^{1,2}([0,T];\mathbb{R})$ such that $k_s=0$ for $s \geq \sigma \wedge S$, $k_0=1$, $l_s = 1$ for $s \leq \sigma \wedge S$, $l_s = 0$ for $s \geq \tau \wedge T$, $\int_0^{\sigma \wedge S} |\dot{k}_s|^2 ds\in L^1$ and $\int_{\sigma \wedge S}^{\tau \wedge T} |\dot{l}_s|^2 ds \in L^1$. Then
\begin{equation}
\begin{split}
&L(P^V_Tf)(x_0) = \\
&\E\left[ \V_T f(x_T)\mathbf{1}_{\lbrace T<\zeta(x_0)\rbrace}\int_0^T \langle \dot{k}_s Z,//_s dB_s\rangle \right]\\
& +\frac{1}{2}\E\left[ \V_T f(x_T)\mathbf{1}_{\lbrace T<\zeta(x_0)\rbrace}\left( \int_0^T \<  W_s \dot{l}_s ,//_s dB_s\> + dV_{T-s}(W_s l_s)ds\right)\int_0^T \dot{k}_s W_s^{-1} //_s dB_s\right].
\end{split}
\end{equation}
\end{theorem}

\begin{proof}
Modifying the process $l_t$ to $l^\epsilon_t$ as in the proof of Theorem \ref{thm:locformone}, it follows from Lemma \ref{lem:locmartonehalf}, the strong Markov property, the boundedness of $P^V_tf$ on $[0,T] \times \overline{D}_2$ and the boundedness of $dP_t^Vf$ and $L P_t^Vf$ on $[\epsilon,T] \times \overline{D}_2$ that the formula
\begin{equation}
\begin{split}
&L(P^V_Tf)(x_0)\\
=\text{ }& \E\left[ \V_T f(x_T)\mathbf{1}_{\lbrace T<\zeta(x_0)\rbrace}\left( \int_0^T \langle \dot{k}_s Z,//_s dB_s\rangle - \int_0^T k_s \left( dV_{T-s}(//_s dB_s) + L V_{T-s} ds\right) \right) \right]\\
\text{ }& +\frac{1}{2}\E\left[\V_T f(x_T)\mathbf{1}_{\lbrace T<\zeta(x_0)\rbrace}\left( \int_0^T \<  W_s \dot{l}_s ,//_s dB_s\> + dV_{T-s}(W_s l_s)ds\right)\int_0^T \dot{k}_s W_s^{-1} //_s dB_s\right]
\end{split}
\end{equation}
holds with $l^\epsilon_t$ in place of $l_t$. The formula also holds as stated, in terms of $l_t$, by taking $\epsilon \downarrow 0$. Applying the It\^{o} formula yields
\begin{equation}
\int_0^T k_s \left( dV_{T-s}(//_s dB_s) + L V_{T-s} ds\right)  = -V_T(x_0) +\int_0^T  (k_s \dot{V}_{T-s}(x_s) -\dot{k}_s V_{T-s}(x_s))ds
\end{equation}
and therefore
\begin{equation}
\begin{split}
&(L-V_T(x_0))(P^V_Tf)(x_0)=\\
& \E\left[\V_T f(x_T)\mathbf{1}_{\lbrace T<\zeta(x_0)\rbrace} \left( \int_0^T \langle \dot{k}_s Z,//_s dB_s\rangle - \int_0^T  (k_s \dot{V}_{T-s}(x_s) -\dot{k}_s V_{T-s}(x_s))ds \right) \right]\\
\text{ }& +\frac{1}{2}\E\left[\V_T f(x_T)\mathbf{1}_{\lbrace T<\zeta(x_0)\rbrace}\left( \int_0^T \<  W_s \dot{l}_s ,//_s dB_s\> + dV_{T-s}(W_s l_s)ds\right)\int_0^T \dot{k}_s W_s^{-1} //_s dB_s\right].
\end{split}
\end{equation}
The result follows from this by Lemma \ref{lem:potlemma}.
\end{proof}

\subsection{Hessian}\label{ss:secderlocform}

For each $w \in T_{x_0}M$ define an operator-valued process $W^{\prime}_t(\cdot,w) : T_{x_0}M \rightarrow T_{x_t}M$ by
\begin{equation}
\begin{split}
W^{\prime}_s(\cdot,w):=\text{ }& W_s \int_0^s W_r^{-1} R(//_r dB_r,W_r(\cdot))W_r (w)\\
&-\frac{1}{2}  W_s \int_0^s W_r^{-1}(\nabla {\Ric}^\sharp_Z + d^\star R -2R(Z))(W_r (\cdot),W_r(w))dr.
\end{split}
\end{equation}
Here the operator $R(Z)$ is defined by $R(Z)(v_1,v_2) := R(Z,v_1)v_2$ and the operator $d^\star R$ is defined by $d^\star R(v_1)v_2 := -\tr \nabla_\cdot R(\cdot,v_1)v_2$ and satisfies
\begin{equation}
\langle d^\star R(v_1)v_2,v_3\rangle =  \langle (\nabla_{v_3} {\Ric}^\sharp) (v_1),v_2\rangle- \langle (\nabla_{v_2} {\Ric}^\sharp) (v_3),v_1\rangle
\end{equation}
for all $v_1,v_2,v_3 \in T_xM$ and $x \in M$. The process $W^{\prime}_t(\cdot,w)$ is the solution to the covariant It\^{o} equation
\begin{equation}
\begin{split}
DW^{\prime}_t(\cdot,w) =\text{ }& R(//_tdB_t,W_t( \cdot ))W_t(w)\\
& -\frac{1}{2} \left(d^\star R -2R(Z)+\nabla {\Ric}^\sharp_Z\right) (W_t( \cdot),W_t(w))dt\\
&\quad-\frac{1}{2} {\Ric}^\sharp_Z (W^{\prime}_t(\cdot,w))dt
\end{split}
\end{equation}
with $W^{\prime}_0(\cdot,w) = 0$. As in the previous section, suppose $D_1$ and $D_2$ are regular domains with $x_0 \in D_1$ and $\overline{D_1}\subset D_2$. Denote by $\sigma$ and $\tau$ the first exit times of $x_t$ from $D_1$ and $D_2$, respectively.

\begin{lemma}\label{lem:locmarttwo}
Suppose $v,w \in T_{x_0}M$, $0<S<T$ and that $k,l$ are bounded adapted processes with paths in the Cameron-Martin space $L^{1,2}([0,T];\Aut(T_{x_0}M))$ such that $k_s=0$ for $s \geq \sigma \wedge S$, $l_s = 1$ for $s \leq \sigma \wedge T$ and $l_s = 0$ for $s \geq \tau \wedge (T-\epsilon)$. Then
\begin{align}
&\V_t(\nabla df_t)(W_t(k_tv),W_t(w)) + \V_t(df_t)(W^{\prime}_t (k_tv,w))- \V_tf_t(x_t)\int_0^t\langle  W^{\prime}_s(\dot{k}_sv,w), //_s dB_s\rangle\\
&\quad- \int_0^t \V_sf_s (x_s)((\nabla d V_{T-s})(W_s(k_sv),W_s(w))+(dV_{T-s})(W^{\prime}_s(k_sv,w)))ds\\
&\quad\quad+\mathbb{V}_tf_t(x_t)\int_{0}^t \< W_s( \dot{l}_sw),//_s dB_s\>\int_0^{t} \<W_s (\dot{k}_sv),//_sdB_s\>\\
&\quad\quad\quad-\V_tdf_t(W_t(l_tw))\int_0^t \<W_s (\dot{k}_sv),//_sdB_s\> \\
&\quad\quad\quad\quad +\int_{0}^{t} \V_s f_s(x_s) dV_{T-s}( W_s((l_s-1)w)) ds\int_0^{t} \<W_s (\dot{k}_sv),//_sdB_s\> \\
&\quad\quad\quad\quad\quad+\int_0^t \V_sf_s(x_s) dV_{T-s}(W_s(w))\int_0^s \langle W_r(\dot{k}_rv),//_r dB_r\rangle ds\\
&\quad\quad\quad\quad\quad\quad -2\int_0^t \V_s(df_s\odot dV_{T-s})(W_s(k_sv),W_s(w))ds \tag{2}
\end{align}
is a local martingale on $[0,\tau \wedge T)$.
\end{lemma}

\begin{proof}
Setting
\begin{equation}
N^{\prime}_t(v,w) := (\nabla d f_t)(W_t(v),W_t(w)) + (df_t)(W^{\prime}_t (v,w))
\end{equation}
and
\begin{equation}
R^{\sharp,\sharp}_x(v_1,v_2) := R_x(\cdot,v_1,v_2,\cdot)^{\sharp}\in T_xM\otimes T_xM
\end{equation}
we see by It\^{o}'s formula and the relations
\begin{equation}
\begin{split}
d \Delta f =& \tr \nabla^2 df - df({\Ric}^\sharp)\\
dZf = & \nabla_Z df + df(\nabla Z)\\
dVf = & f dV + V df\\
\nabla d (\Delta f) =& \tr \nabla^2 (\nabla d f) - 2(\nabla d f)({\Ric}^\sharp \odot \id -  R^{\sharp,\sharp}) - df(d^\star R + \nabla {\Ric}^\sharp)\\
\nabla d (Zf) =& \nabla_Z (\nabla d f) + 2(\nabla d f)(\nabla Z \odot \id) + df(\nabla \nabla Z + R(Z))\\
\nabla d (V_t f) =& f \nabla d V_t + 2df \odot dV_t + V_t \nabla d f
\end{split}
\end{equation}
(the fourth one is a consequence of the differential Bianchi identity; see \cite[p.~219]{ChowHamilton}, and the fifth one a consequence of the Ricci identity) that
\begin{equation}
\begin{split}
&dN^{\prime}_t(v,w)\\
 =\text{ }& (\nabla_{//_t dB_t} \nabla d f_t)(W_t(v),W_t(w)) + (\nabla d f_t)\left(\frac{D}{dt}W_t(v),W_t(w)\right)dt \\
&+ (\nabla d f_t)\left(W_t(v),\frac{D}{dt}W_t(w)\right)dt\\
&+ \partial_t (\nabla d f_t)(W_t(v),W_t(w))dt + \left(\frac{1}{2}\tr \nabla^2 + \nabla_Z\right) (\nabla d f_t)(W_t(v),W_t(w))dt\\
&+ (\nabla_{//_tdB_t} df_t)(W^{\prime}_t(v,w)) + (df_t)\left(DW^{\prime}_t(v,w)\right) + \langle d(df_t),DW^{\prime}_t(v,w)\rangle\\
& + \partial_t (df_t)(W^{\prime}_t(v,w))dt + \left(\frac{1}{2}\tr \nabla^2 + \nabla_Z\right)(df_t)(W^{\prime}_t(v,w))dt\\
\stackrel{m}{=}\text{ }&  f_t(x_t) (\nabla d V_{T-t})(W_t(v),W_t(w))dt + f_t(x_t) (dV_{T-t})(W^{\prime}_t(v,w))dt \\[7pt]
&+ 2(df_t\odot dV_{T-t})(W_t(v),W_t(w))dt + V_{T-t} N^{\prime}_t(v,w)dt
\end{split}
\end{equation}
for which we calculated
\begin{equation}
\left[ d(df),DW^{\prime}(v,w)\right]_t = (\nabla d f_t)(R^{\sharp,\sharp}(W_t(v),W_t(w)))dt.
\end{equation}
It follows that
\begin{equation}
\begin{split}
d(\V_tN^{\prime}_t(k_t v,w)) \stackrel{m}{=}\text{ }& \V_tf_t(x_t) ((\nabla d V_{T-t})(W_t(k_t v),W_t(w))+(dV_{T-t})(W^{\prime}_t(k_tv,w)))dt\\
&+\V_tN_t(\dot{k}_tv,w)dt+ 2\V_t(df_t\odot dV_{T-t})(W_t(k_tv),W_t(w))dt
\end{split}
\end{equation}
so that
\begin{equation}
\begin{split}
&\V_tN^{\prime}_t(k_t v,w) - \int_0^t \V_s (\nabla d f_s)(W_s(\dot{k}_s v),W_s(w)) ds - \int_0^t \V_s (df_s)(W^{\prime}_s (\dot{k}_s v,w))ds\\
&\quad-2\int_0^t\V_s(df_s \odot dV_{T-s})(W_s(k_sv),W_s(w))ds\\
&\quad\quad- \int_0^t \V_sf_s(x_s) ((\nabla d V_{T-s})(W_s(k_sv),W_sw))+(dV_{T-s})(W^{\prime}_s(k_sv,w)))ds
\end{split}
\end{equation}
is a local martingale. By the formula
\begin{equation}
\V_t f_t(x_t) = f_0(x_0) +\int_0^t \V_s df_s(//_sdB_s)
\end{equation}
and integration by parts we see that
\begin{equation}
\int_0^t\V_s(df_s)(W^{\prime}_s(\dot{k}_sv,w)) ds -\V_tf_t(x_t)\int_0^t\langle  W^{\prime}_s(\dot{k}_s v,w), //_s dB_s\rangle
\end{equation}
is a local martingale. Similarly, by the formula
\begin{equation}
\V_t df_t(W_t) = df_0 + \int_0^t \V_s(\nabla d f_s)(//_sdB_s,W_s) + \int_0^t \V_s f_s(x_s) dV_{T-s}(W_s)ds
\end{equation}
and integration by parts we see that
\begin{equation}
\begin{split}
&\int_0^t \V_s (\nabla d f_s)(W_s(\dot{k}_s v),W_s(w)) ds-\V_tdf_t(W_t(w))\int_0^t \langle W_s(\dot{k}_sv),//_sdB_s\rangle \\
&\quad+\int_0^t \V_sf_s(x_s) dV_{T-s}(W_s(w))\int_0^s \langle W_r(\dot{k}_rv),//_r dB_r\rangle ds
\end{split}
\end{equation}
is yet another local martingale. Therefore
\begin{equation}
\begin{split}
&\V_t(\nabla d f_t)(W_t(k_tv),W_t(w)) + \V_t(df_t)(W^{\prime}_t (k_tv,w))\\
&\quad- \int_0^t \V_sf_s(x_s) ((\nabla d V_{T-s})(W_s(k_sv),W_s(w))+(dV_{T-s})(k_s W^{\prime}_s(v,w)))ds\\
&\quad\quad - \V_tf_t(x_t)\int_0^t\langle  W^{\prime}_s(\dot{k}_s v,w), //_s dB_s\rangle-2\int_0^t \V_s(df_s\odot dV_{T-s})(W_s(k_sv),W_s(w))ds\\
&\quad\quad\quad+\int_0^t \V_sf_s(x_s) dV_{T-s}(W_s(w))\int_0^s \langle W_r(\dot{k}_rv),//_r dB_r\rangle ds\\
&\quad\quad\quad\quad -\V_tdf_t(W_t(w))\int_0^t \langle W_s(\dot{k}_sv),//_sdB_s\rangle
\end{split}
\end{equation}
is a local martingale. By Lemma \ref{lem:locmartone} it follows that
\begin{eqnarray}
O^1_t& =& \V_tdf_t(W_t((l_t-1)w)) -\mathbb{V}_tf_t(x_t)\int_0^t \< W_s( \dot{l}_sw),//_s dB_s\> \\
& &\quad -\int_0^t \V_s f_s(x_s) dV_{T-s}( W_s((l_s-1)w)) ds,\\
O^2_t &= & \int_0^t \<W_s (\dot{k}_sv),//_sdB_s\>
\end{eqnarray}
are two local martingales. So the product $O^1_tO^2_t$ is also a local martingale, since $O^1 = 0$ on $[0,\sigma \wedge S]$ with $O^2$ constant on $[\sigma \wedge S,\tau \wedge (T-\epsilon))$. Applying this fact to the previous equation completes the proof.
\end{proof}

\begin{theorem}\label{thm:locformtwo}
Suppose $x_0 \in D_1$ with $v,w \in T_{x_0}M$, $f \in \mathcal{B}_b$, $V$ bounded below and $0<S<T$. Assume $k,l$ are bounded adapted processes with paths in the Cameron-Martin space $L^{1,2}([0,T];\Aut(T_{x_0}M))$ such that $k_s=0$ for $s \geq \sigma \wedge S$, $k_0=1$, $l_s = 1$ for $s \leq \sigma \wedge S$, $l_s = 0$ for $s \geq \tau \wedge T$, $\int_0^{\sigma \wedge S} |\dot{k}_s|^2 ds\in L^1$ and $\int_{\sigma \wedge S}^{\tau \wedge T} |\dot{l}_s|^2 ds \in L^1$. Then
\begin{equation}
\begin{split}
&(\nabla d P^V_Tf)(v,w) \\
=\text{ }&  - \E \left[ \V_T f(x_T)\mathbf{1}_{\lbrace T < \zeta(x_0)\rbrace}\int_0^T\langle  W^{\prime}_s(\dot{k}_sv,w), //_s dB_s\rangle\right]\\
&-\E\left[ \V_T f(x_T)\mathbf{1}_{\lbrace T < \zeta(x_0)\rbrace}\int_0^T ((\nabla d V_{T-s})(W_s(k_sv),W_s(w))+(dV_{T-s})(W^{\prime}_s(k_sv,w)))ds\right]\\
&+\E\left[ \V_T f(x_T)\mathbf{1}_{\lbrace T < \zeta(x_0)\rbrace}\left(\int_{0}^T \< W_s( \dot{l}_sw),//_s dB_s\> + dV_{T-s}( W_s(l_sw)) ds\right)\right.\\
&\quad\quad\quad\quad\quad\quad\quad\quad\quad\quad\quad\quad\quad\quad\quad\left.\cdot\left(\int_0^{T} \<W_s (\dot{k}_sv),//_sdB_s\>+ dV_{T-s}(W_s(k_sv))ds\right)\right].
\end{split}
\end{equation}
\end{theorem}

\begin{proof}
Modifying the process $l_t$ to $l^\epsilon_t$ as in the proof of Theorem \ref{thm:locformone}, it follows from Lemma \ref{lem:locmarttwo}, the strong Markov property, the boundedness of $P^V_tf$ on $[0,T] \times \overline{D}_2$ and the boundedness of $dP_t^Vf$ and $\nabla d P_t^Vf$ on $[\epsilon,T] \times \overline{D}_2$ that the formula
\begingroup
\allowdisplaybreaks
\begin{equation}
\begin{split}
&(\nabla d P^V_Tf)(v,w) \\
=\text{ }&  - \E \left[ \V_T f(x_T)\mathbf{1}_{\lbrace T < \zeta(x_0)\rbrace}\int_0^{T}\langle  W^{\prime}_s(\dot{k}_sv,w), //_s dB_s\rangle\right]\\
&+\E\left[ \V_T f(x_T)\mathbf{1}_{\lbrace T < \zeta(x_0)\rbrace}\int_{0}^T \< W_s( \dot{l}_sw),//_s dB_s\>\int_0^{T} \<W_s (\dot{k}_sv),//_sdB_s\>\right]\\
&-\E\left[\V_Tf(x_T)\mathbf{1}_{\lbrace T < \zeta(x_0)\rbrace} \int_0^{T} ((\nabla d V_{T-s})(W_s(k_sv),W_s(w))+(dV_{T-s})(W^{\prime}_s(k_sv,w)))ds\right]\\
&+\E\left[ \V_T f(x_T)\mathbf{1}_{\lbrace T < \zeta(x_0)\rbrace} \int_{0}^T dV_{T-s}( W_s(l_sw)) ds\int_0^{T} \<W_s (\dot{k}_sv),//_sdB_s\>\right] \\
&-\E\left[\V_T f(x_T)\mathbf{1}_{\lbrace T < \zeta(x_0)\rbrace}\int_0^{T} \left(\int_0^s dV(W_r(w))dr\right) \langle W_s(\dot{k}_sv),//_s dB_s\rangle  \right]\\
& -\E\left[ \int_0^{T} \V_sdf_s(W_s(k_s v))dV_{T-s}(W_s(w))ds\right]\\
& -\E\left[ \int_0^{T} \V_sdf_s(W_s(w))dV_{T-s}(W_s(k_s v))ds\right]
\end{split}
\end{equation}
\endgroup
holds with $l^\epsilon_t$ in place of $l_t$, and therefore in terms of $l_t$ by taking $\epsilon \downarrow 0$. Paying close attention to the assumptions on $l$ and $k$, it follows from this, by Theorem \ref{thm:locformone} and the strong Markov property, that
\begingroup
\allowdisplaybreaks
\begin{align}
& -\mathbb{E}\left[ \int_0^{T} \V_sdf_s(W_s (w))dV_{T-s}(W_s(k_s v))ds \right]\\
=\text{ }& +\mathbb{E}\left[ \V_T f(x_T)\mathbf{1}_{\lbrace T < \zeta(x_0)\rbrace} \int_0^{T} \int_r^{T} \<W_s (\dot{l}_sw) ,//_s dB_s\>dV_{T-r}(W_r(k_rv))dr \right]\\
\text{ } & - \mathbb{E}\left[\V_T f(x_T)\mathbf{1}_{\lbrace T < \zeta(x_0)\rbrace}\int_0^{T} \left( \int_0^r dV_{T-u}(W_u (w)) du\right)dV_{T-r}(W_r(k_rv))dr \right]\\
\text{ }&+ \mathbb{E}\left[ \V_T f(x_T)\mathbf{1}_{\lbrace T < \zeta(x_0)\rbrace}  \int_0^{T}dV_{T-r}(W_r(k_rv))dr \int_0^{T} dV_{T-s}(W_s( l_sw))ds \right]\\
=\text{ }& +\mathbb{E}\left[ \V_T f(x_T)\mathbf{1}_{\lbrace T < \zeta(x_0)\rbrace} \int_{0}^{T} \<W_s (\dot{l}_sw) ,//_s dB_s\>  \int_0^{T} dV_{T-r}(W_r(k_rv))dr \right]\\
\text{ } & - \mathbb{E}\left[\V_T f(x_T)\mathbf{1}_{\lbrace T < \zeta(x_0)\rbrace} \int_0^{T} \left( \int_0^r dV_{T-u}(W_u (w)) du\right)dV_{T-r}(W_r(k_rv))dr \right]\\
\text{ }&+ \mathbb{E}\left[ \V_T f(x_T)\mathbf{1}_{\lbrace T < \zeta(x_0)\rbrace} \int_0^{T}dV_{T-r}(W_r(k_rv))dr \int_0^{T} dV_{T-s}(W_s( l_sw))ds \right]
\end{align}
\endgroup
from which it follows that
\begingroup
\allowdisplaybreaks
\begin{equation}
\begin{split}
&(\nabla d P^V_Tf)(v,w) \\
=\text{ }&  - \E \left[ \V_T f(x_T)\mathbf{1}_{\lbrace T < \zeta(x_0)\rbrace}\int_0^T\langle  W^{\prime}_s(\dot{k}_sv,w), //_s dB_s\rangle\right]\\
&+\E\left[ \V_T f(x_T)\mathbf{1}_{\lbrace T < \zeta(x_0)\rbrace}\int_{0}^T \< W_s( \dot{l}_sw),//_s dB_s\>\int_0^{T} \<W_s (\dot{k}_sv),//_sdB_s\>\right]\\
&-\E\left[\V_Tf(x_T)\mathbf{1}_{\lbrace T < \zeta(x_0)\rbrace} \int_0^T ((\nabla d V_{T-s})(W_s(k_sv),W_s(w))+(dV_{T-s})(W^{\prime}_s(k_sv,w)))ds\right]\\
&+\E\left[ \V_T f(x_T)\mathbf{1}_{\lbrace T < \zeta(x_0)\rbrace} \int_{0}^T dV_{T-s}( W_s(l_sw)) ds\int_0^{T} \<W_s (\dot{k}_sv),//_sdB_s\>\right] \\
& + \mathbb{E}\left[ \V_T f(x_T) \mathbf{1}_{\lbrace T < \zeta(x)\rbrace}\int_{0}^{T} \<W_s \dot{l}_sw ,//_s dB_s\>  \int_0^{T} dV_{T-r}(W_r(k_rv))dr \right]\\
\text{ }&+ \mathbb{E}\left[ \V_T f(x_T) \mathbf{1}_{\lbrace T < \zeta(x)\rbrace} \int_0^{T}dV_{T-r}(W_r(k_rv)) dr \int_0^{T} dV_{T-s}( W_s (l_sw))ds \right]\\
&-\E\left[\V_T f(x_T)\mathbf{1}_{\lbrace T < \zeta(x)\rbrace} \int_0^{T}\left( \int_0^s dV_{T-r}(W_r(w))dr\right)\langle W_s(\dot{k}_sv),//_s dB_s\rangle \right]\\
&-\E\left[\V_T f(x_T)\mathbf{1}_{\lbrace T < \zeta(x)\rbrace} \int_0^{T}\left( \int_0^s dV_{T-r}(W_r(w))dr\right)dV_{T-s}(W_s(k_sv))ds\right]\\
& -\E\left[ \int_0^{T} \V_sdf_s(W_s(k_s v))dV_{T-s}(W_s(w))ds\right].
\end{split}
\end{equation}
\endgroup
Finally, by the stochastic Fubini theorem \cite[Theorem~2.2]{Veraar} we have
\small
\begin{equation}
\begin{split}
&\E\left[\V_T f(x_T)\mathbf{1}_{\lbrace T < \zeta(x_0)\rbrace}\int_0^{T} \int_0^s dV_{T-r}(W_r(w))dr (\langle W_s(\dot{k}_sv),//_s dB_s\rangle +dV_{T-s}(W_sk_sv)ds)\right]\\
=\text{ }&\E\left[\V_T f(x_T)\mathbf{1}_{\lbrace T < \zeta(x_0)\rbrace}\int_0^{T} \left(\int_s^{T} \langle W_r(\dot{k}_rv),//_r dB_r\rangle +dV_{T-r}(W_rk_rv)dr\right) dV_{T-s}(W_s(w))ds\right]\\
\end{split}
\end{equation}
\normalsize
which cancels the final three terms in the previous equation, by the strong Markov property, Theorem \ref{thm:locformone} and the assumptions on $k$.
\end{proof}

For the case $Z=0$ and $V=0$, Theorem \ref{thm:locformtwo} reduces to \cite[Theorem~2.1]{APT}.

\begin{remark}\label{rem:notsmooth}
We have assumed that $V$ is bounded below and smooth. However, so long as $V$ is bounded below and continuous with $V_t \in C^1$ for each $t \in [0,T]$ and $P^V f\in C^{1,3}([\epsilon,T]\times M)$ then the results of Subsection \ref{ss:firstderlocform} evidently remain valid. Similarly, the results of Subsection \ref{ss:secderlocformhalf} evidently remain valid if $V$ is bounded below, $C^1$ with $V_t \in C^2$ for each $t \in [0,T]$ and $P^V f\in C^{1,4}([\epsilon,T]\times M)$. Similarly, the results of Subsection \ref{ss:secderlocform} evidently remain valid if $V$ is bounded below and continuous with $V_t \in C^2$ for each $t \in [0,T]$ and $P^V f\in C^{1,4}([\epsilon,T]\times M)$.
\end{remark}

\section{Stationary Solutions}\label{sec:statsolns}

Now suppose $\phi\in C^2(D) \cap C(\overline{D})$ solves the eigenvalue equation
\begin{equation}
(L - V)\phi = -E \phi
\end{equation}
on the regular domain $D$, for some $E \in \mathbb{R}$ and a function $V\in C^2$ which does not depend on time and which is bounded below. Denoting by $\tau$ the first exit time from $D$ of the diffusion $x_t$ with generator $L$ and assuming $x_0 \in D$, one has, in analogy to the Feynman-Kac formula \eqref{eq:fkform}, the formula
\begin{equation}
\phi(x_0) = \E \left[ \V_\tau \phi(x_\tau) e^{E\tau}\right].
\end{equation}
Furthermore, the methods of the previous section can easily be adapted to find formulae for the derivatives of $\phi$. In particular, one simply sets $f_t = \phi$, replaces $V_{T-t}$ with $V-E$ and the calculations carry over almost verbatim (although there is no application of the strong Markov property; in this case the local martingale property is enough). In particular, for the derivative $d\phi$, supposing $k$ is a bounded adapted process with paths in the Cameron-Martin space $L^{1,2}([0,\infty),\Aut(T_{x_0}M))$ with $k_0=1$, $k_t=0$ for $t \geq \tau$ and $\int_0^{\tau} |\dot{k}_s|^2 ds\in L^1$, one obtains
\begin{equation}
(d\phi)(v)  = - \mathbb{E}\left[  \V_\tau \phi(x_\tau)e^{E \tau} \int_0^{\tau} \langle W_s (\dot{k}_s v),//_sdB_s\rangle+dV(W_s (k_sv)) ds \right]
\end{equation}
for each $v \in T_{x_0}M$. When $V=0$ and $E=0$ this formulae reduces to the one given in \cite{Thalmaier97}. Similarly, denoting by $D_1$ a regular domain with $x_0 \in D_1$ and $\overline{D_1}\subset D $ and by $\sigma$ the first exit time of $x_t$ from $D_1$, supposing $k,l$ are bounded adapted processes with paths in the Cameron-Martin space $L^{1,2}([0,\infty);\Aut(T_{x_0}M))$ such that $k_s=0$ for $s \geq \sigma$, $k_0=1$, $l_s = 1$ for $s \leq \sigma$, $l_s = 0$ for $s \geq \tau$, $\int_0^{\sigma} |\dot{k}_s|^2 ds\in L^1$ and $\int_{\sigma }^{\tau} |\dot{l}_s|^2 ds \in L^1$, for the Hessian of $\phi$ one obtains
\begin{equation}
\begin{split}
&(\nabla d \phi)(v,w) \\
=\text{ }&  - \E \left[ \V_\sigma \phi(x_\sigma)e^{E \sigma}\int_0^\sigma \langle  W^{\prime}_s(\dot{k}_sv,w), //_s dB_s\rangle\right]\\
&-\E\left[  \V_\sigma \phi(x_\sigma)e^{E \sigma}\int_0^\sigma ((\nabla d V)(W_s(k_sv),W_s(w))+(dV)(W^{\prime}_s(k_sv,w)))ds\right]\\
&+\E\left[ \V_\tau \phi(x_\tau)e^{E \tau}\left(\int_{0}^\tau \< W_s( \dot{l}_sw),//_s dB_s\> + dV( W_s(l_sw)) ds\right)\right.\\
&\quad\quad\quad\quad\quad\quad\quad\quad\quad\quad\quad\quad\quad\quad\quad\left.\cdot\left(\int_0^\sigma \<W_s (\dot{k}_sv),//_sdB_s\>+ dV(W_s(k_sv))ds\right)\right]
\end{split}
\end{equation}
for all $v,w \in T_{x_0}M$.

\section{Local and Global Estimates}\label{sec:localest}

\subsection{Gradient}\label{ss:firstderlocest}

\begin{theorem}\label{thm:locestone}
Suppose $D_0,D$ are regular domains with $x_0 \in \overline{D_0} \subset D$, $V$ bounded below and $T>0$. Set
\begin{equation}
\begin{split}
\underline{\kappa}_D :=\text{ }& \inf \lbrace {\Ric}_Z(v,v):v \in T_y M,y \in D,|v|=1\rbrace;\\
v_D:=\text{ }& \sup \lbrace |(dV_t)_y(v)|: v \in T_y M,y \in D,|v|=1, t \in [0,T]\rbrace.\\
\end{split}
\end{equation}
Then there exists a positive constant $$C\equiv C(n,T,\inf V,\underline{\kappa}_D,v_D,d(\partial D_0,\partial D))$$ such that
\begin{equation}\label{eq:locesttime}
|dP^V_tf_{x_0}| \leq \frac{C}{\sqrt{t}}|f|_\infty
\end{equation}
for all $0 < t \leq T$, $x_0 \in D_0$ and $f \in \mathcal{B}_b$.
\end{theorem}

\begin{proof}
According to \cite{APT}, the process $k_t$ appearing in Theorem \ref{thm:locformone} can be chosen so that $|k_s| \leq c(T)$ for all $s \in [0,T]$, almost surely, with
\begin{equation}
\mathbb{E}\left[ \int_0^{T} |\dot{k}_s|^2 ds\right]^\frac{1}{2}\leq \frac{\tilde{C}}{\sqrt{1-e^{-\tilde{C}^2T}}}
\end{equation}
for a positive constant $\tilde{C}$ which depends continuously on $\underline{\kappa}$, $n$ and $d(\partial D_0,\partial D)$. The details of this can be found in \cite{ThalmaierWanggradest}. By Theorem \ref{thm:locformone} and the Cauchy-Schwarz inequality, using equation \eqref{eq:eqfordpt} and the parameter $\underline{\kappa}_D$ to control the size of the damped parallel transport, we have
\begin{equation}
\begin{split}
|dP^V_Tf| \leq \text{ }& |f|_\infty e^{-\inf V} \left(\mathbb{E}\left[ \mathbf{1}_{\lbrace T < \zeta(x_0)\rbrace} \int_0^{T} | W_s|^2 |\dot{k}_s|^2 ds\right]^{\frac{1}{2}}\right.\\
&\quad\quad\quad\quad\quad\quad\left. + \mathbb{E}\left[\mathbf{1}_{\lbrace T < \zeta(x_0)\rbrace} \int_0^{T} |dV_{T-s}||W_s| |k_s| ds \right]\right)\\
\leq \text{ }& |f|_\infty e^{( -\inf V-\frac{1}{2}(\underline{\kappa}_D \wedge 0)) T} \left( \mathbb{E}\left[ \int_0^{T} |\dot{k}_s|^2 ds\right]^{\frac{1}{2}}+  v_D \E\left[\int_0^T |k_s| ds\right]\right)
\end{split}
\end{equation}
so the estimate \eqref{eq:locesttime} follows by substituting the bounds on $k$ and $\dot{k}$.
\end{proof}

Note that in the above theorem, the dependence of the constant $C$ on $d(D_0,D)$ is such that if one tries to shrink $D$ onto $D_0$, so as to reduce the information needed about $\Ric_Z$ and $dV$, then the constant blows up at rate $1/d(D_0,D)$. There is therefore a trade-off between the size of the domain and the size of the constant. This behaviour is unavoidable and also occurs with respect to the domains $D_0,D_1$ and $D_2$ which appear in Theorems \ref{thm:locestonehalf} and \ref{thm:locesttwo} below.

\begin{corollary}\label{cor:globestone}
Suppose ${\Ric}_Z$ is bounded below with $|dV|$ bounded and $V$ bounded below. Then for all $T>0$ there exists a positive constant $C\equiv C(n,T)$ such that
\begin{equation}
|dP^V_tf_{x}| \leq \frac{C}{\sqrt{t}}|f|_\infty
\end{equation}
for all $0 < t \leq T$, $x \in M$ and $f \in \mathcal{B}_b$.
\end{corollary}

\begin{proof}
As explained in the proof of Theorem \ref{thm:locestone}, the dependence on $D_0$ of the constant appearing there is via the quantity $d(\partial D_0,\partial D)$. If $M$ is compact then the injectivity radius $\inj(M)$ is positive and we can choose $D_0= B_{\inj(M)/4}(x_0)$ and $D= B_{\inj(M)/2}(x_0)$, in which case $d(\partial D_0,\partial D) =\inj(M)/4$. Conversely, if $M$ is non-compact then for each $x_0 \in M$ there exist $D_0,D$ with $x_0 \in \overline{D_0} \subset D$ and $d(\partial D_0,\partial D) = 1$. Consequently, the result follows from Theorem \ref{thm:locestone}.
\end{proof}

\subsection{Generator}

\begin{theorem}\label{thm:locestonehalf}
Suppose $D_0,D_1$ and $D_2$ are regular domains with $x_0 \in \overline{D_0} \subset D_1$, $\overline{D_1} \subset D_2$, $V$ bounded below and $T>0$. Set
\begin{equation}
\begin{split}
\kappa_{D_2} :=\text{ }& \sup \lbrace |{\Ric}_Z(v,v)|: v \in T_y M,y \in D_2,|v|=1\rbrace;\\
v_{D_2}:=\text{ }& \sup \lbrace |(dV_t)_y(v)|: v \in T_y M,y \in D_2,|v|=1, t \in [0,T]\rbrace;\\
z_{D_1}:=\text{ }& \sup \lbrace |Z|_y: y \in D_1 \rbrace.
\end{split}
\end{equation}
Then there exists a positive constant $$C\equiv C(n,T,\inf V,\kappa_{D_2},v_{D_2},z_{D_1},d(\partial D_0,\partial D_1),d(\partial D_0,\partial D_2))$$ such that
\begin{equation}
|L P^V_tf_{x_0}| \leq \frac{C}{t}|f|_\infty
\end{equation}
for all $0 < t \leq T$, $x_0 \in D_0$ and $f \in \mathcal{B}_b$.
\end{theorem}

\begin{proof}
According to \cite{APT}, the processes $k_t$ and $l_t$ appearing in Theorem \ref{thm:formhalf} can be chosen so that
\begin{equation}
\begin{split}
&|k_s| \leq c_1(n,\kappa_{D_1},T,d(\partial D_0,\partial D_1)),\\
&|l_s| \leq c_2(n,\kappa_{D_2},T,d(\partial D_0,\partial D_2))
\end{split}
\end{equation}
for all $s \in [0,T]$, almost surely, with
\begin{equation}
\mathbb{E}\left[ \int_0^{T} |\dot{k}_s|^2 ds\right]^\frac{1}{2}\leq \frac{\tilde{C}_1}{\sqrt{1-e^{-\tilde{C}_1^2T}}}, \quad \mathbb{E}\left[ \int_0^{T} |\dot{l}_s|^2 ds\right]^\frac{1}{2}\leq \frac{\tilde{C}_2}{\sqrt{1-e^{-\tilde{C}_2^2T}}}
\end{equation}
for positive constants $\tilde{C}_1$ and $\tilde{C}_2$ which depend continuously on $\kappa$, $n$ and on $d(\partial D_0,\partial D_1)$ and $d(\partial D_0,\partial D_2)$, respectively. By Theorem \ref{thm:formhalf} and the Cauchy-Schwarz inequality we have
\begin{equation}
\begin{split}
&|L(P^V_Tf)(x_0)| \\
\leq\text{ }& e^{-\inf V}|f|_\infty z_{D_1} \E\left[\int_0^T  |\dot{k}_s|^2ds\right]^\frac{1}{2}\\
& +\frac{1}{2}|f|_\infty e^{T(\kappa_{D_2}-\inf V)} \E \left[ \int_0^T |\dot{k}_s|^2ds\right]^\frac{1}{2} \left( \E\left[ \int_0^T |\dot{l}_s|^2ds\right]^\frac{1}{2} + v_{D_2} \E\left[\left(\int_0^T |l_s| ds\right)^2\right]^\frac{1}{2}\right)
\end{split}
\end{equation}
so the result follows by substituting the bounds on $k,\dot{k},l$ and $\dot{l}$.
\end{proof}

\begin{corollary}\label{cor:globesthalf}
Suppose $|{\Ric}_Z|$, $|dV|$, $|Z|$, are bounded with $V$ bounded below. Then there exists a positive constant $C\equiv C(n,T)$ such that
\begin{equation}
|L P^V_tf_{x}| \leq \frac{C}{t}|f|_\infty
\end{equation}
for all $0 < t \leq T$, $x \in M$ and $f \in \mathcal{B}_b$.
\end{corollary}

\begin{proof}
The result follows from Theorem \ref{thm:locestonehalf}, since as in Corollary \ref{cor:globestone} any dependence of the constant on $D_0,D_1$ and $D_2$ can be eliminated.
\end{proof}

\subsection{Hessian}

\begin{theorem}\label{thm:locesttwo}
Suppose $D_0,D_1$ and $D_2$ are regular domains with $x_0 \in \overline{D_0} \subset D_1$, $\overline{D_1} \subset D_2$, $V$ bounded below and $T>0$. Set
\begin{equation}
\begin{split}
\kappa_{D_2} :=\text{ }& \sup \lbrace |{\Ric}_Z(v,v)|: v \in T_y M,y \in D_2,|v|=1\rbrace;\\
v_{D_2}:=\text{ }& \sup \lbrace |(dV_t)_y(v)|: v \in T_y M,y \in D_2,|v|=1, t \in [0,T]\rbrace;\\
v'_{D_1}:=\text{ }& \sup \lbrace |(\nabla d V_t)_y(v,v)|: v \in T_y M,y \in D_1,|v|=1,t\in [0,T]\rbrace;\\
\rho_{D_1}:=\text{ }& \sup \lbrace |R(w,v)v|: v,w \in T_y M,y \in D_1,|v|=|w|=1\rbrace;\\
\rho'_{D_1}:=\text{ }& \sup \lbrace |(\nabla{\Ric}^\sharp_Z+d^\star R -2R(Z))(v,v)|: v \in T_y M,y \in D_1,|v|=1\rbrace.
\end{split}
\end{equation}
Then there exists a positive constant $$C\equiv C(n,T,\inf V,\kappa_{D_2},v_{D_2},v'_{D_1},\rho_{D_1},\rho'_{D_1},d(\partial D_0,\partial D_1),d(\partial D_0,\partial D_2))$$ such that
\begin{equation}
|\nabla d P^V_tf_{x_0}| \leq \frac{C}{t}|f|_\infty
\end{equation}
for all $0 < t \leq T$, $x_0 \in D_0$ and $f \in \mathcal{B}_b$.
\end{theorem}

\begin{proof}
Recalling the defining equation for $W'_s(v,w)$ and choosing the processes $k_t$ and $l_t$ as in the proof of Theorem \ref{thm:locestonehalf}, it follows for the process $k_t$ that
\begin{equation}
\begin{split}
\E \left[ \left(\int_0^T\langle  \dot{k}_s W_s \int_0^s W_r^{-1} R(//_r dB_r,W_r)W_r , //_s dB_s\rangle\right)^2\right]^\frac{1}{2} &\text{ }\leq \frac{\tilde{C}_3 e^{\kappa_{D_1} T}}{\sqrt{1-e^{-\tilde{C}_3^2 T}}},\\
\E \left[ \left(\int_0^T\langle  \dot{k}_s W_s \int_0^s W_r^{-1}(\nabla {\Ric}^\sharp_Z + d^\star R)(W_r ,W_r)dr, //_s dB_s\rangle\right)^2\right]^\frac{1}{2}&\text{ }\leq \frac{\tilde{C}_4e^{\kappa_{D_1} T}}{\sqrt{1-e^{-\tilde{C}_4^2 T}}}
\end{split}
\end{equation}
for positive constants $\tilde{C}_3$ and $\tilde{C}_4$ which depend continuously on $\kappa_{D_1}$, $\rho_{D_1}$, $\rho'_{D_1}$, $n$ and on $d(\partial D_0,\partial D_1)$ and $d(\partial D_0,\partial D_2)$, respectively. The details of this, including explicit bounds on these constants (and on those appearing in Theorems \ref{thm:locestone} and \ref{thm:locestonehalf}) are found in \cite[Section~4.2]{Plank}, with appropriate bounds for the radial part of the diffusion being given as in the proof of \cite[Corollary~2.1.2]{Wangbook}. By Theorem \ref{thm:locformtwo} and the Cauchy-Schwarz inequality we have
\begin{equation}
\begin{split}
&|\nabla d P^V_Tf| \\
\leq\text{ }& |f|_\infty e^{T(\kappa_{D_2}-\inf V)}\left(\frac{\tilde{C}_3}{\sqrt{1-e^{-\tilde{C}_3^2 T}}} + \frac{1}{2}\frac{\tilde{C}_4}{\sqrt{1-e^{-\tilde{C}_4^2 T}}}\right)\\
&+|f|_\infty e^{T(\kappa_{D_2}-\inf V)}\left(v'_{D_1} \E\left[ \left(\int_0^T |k_s|ds\right)^2\right]^{\frac{1}{2}}+v_{D_2} c_1^2 (\rho_{D_1} \vee \frac{1}{2}\rho'_{D_1}) \frac{T^2}{\sqrt{2}}\right)\\
&+|f|_\infty e^{T(\kappa_{D_2}-\inf V)} \left( \E\left[ \int_{0}^T  |\dot{l}_s|^2 ds\right]^\frac{1}{2} + v_{D_2}\E\left[\left(\int_0^{T} |l_s|ds\right)^2 \right]^\frac{1}{2}\right)\\
&\quad\quad\quad\quad\quad\quad\quad\quad\quad\quad\quad\quad\quad\quad\quad\cdot\left( \E\left[\int_0^{T} |\dot{k}_s|^2ds\right]^\frac{1}{2}+v_{D_2}\E\left[\left(\int_0^{T} |k_s|ds\right)^2\right]^\frac{1}{2}\right)
\end{split}
\end{equation}
so the result follows by substituting the bounds on $k,\dot{k},l$ and $\dot{l}$.
\end{proof}

\begin{corollary}\label{cor:globesttwo}
Suppose $|{\Ric}_Z|$, $|dV|$, $|\nabla d V|$, $|\nabla{\Ric}^\sharp_Z+d^\star R-2R(Z)|$, $|R|$ are bounded with $V$ bounded below. Then there exists a positive constant $C\equiv C(n,T)$ such that
\begin{equation}
|\nabla d P^V_tf_{x}| \leq \frac{C}{t}|f|_\infty
\end{equation}
for all $0 < t \leq T$, $x \in M$ and $f \in \mathcal{B}_b$.
\end{corollary}

\begin{proof}
The result follows from Theorem \ref{thm:locesttwo}, since as in Corollaries \ref{cor:globestone} and \ref{cor:globesthalf} any dependence of the constant on $D_0,D_1$ and $D_2$ can be eliminated.
\end{proof}

\section{Non-local Formulae}\label{sec:globform}

If ${\Ric}_Z$ is bounded below then, by \cite[Corollary~2.1.2]{Wangbook}, the diffusion $x_t$ is non-explosive, which is to say $\zeta(x_0)=\infty$, almost surely. While the formulae in this section require non-explosion and global bounds on the various curvature operators, they are expressed in terms of explicit and deterministic processes $k$ and $l$.

\begin{theorem}\label{thm:nonlocformone}
Suppose $x_0\in M$ with $v \in T_{x_0}M$, $f \in \mathcal{B}_b$, $V$ bounded below and $T>0$. Set
\begin{equation}
k_s = \frac{T-s}{T}.
\end{equation}
Suppose ${\Ric}_Z$ is bounded below with $|dV|$ bounded and $V$ bounded below. Then
\begin{equation}
(dP^V_Tf)(v)  = -\mathbb{E}\left[ \V_T f(x_T)\int_0^{T} \langle W_s (\dot{k}_sv),//_sdB_s\rangle + dV_{T-s}(W_s(k_s v))ds\right].
\end{equation}
\end{theorem}

\begin{proof}
It follows from Corollary \ref{cor:globestone} that $|dP^V_t|$ is bounded on $[\epsilon,T] \times M$. Therefore, using
\begin{equation}
k^\epsilon_s = \frac{T-\epsilon-s}{T-\epsilon} \vee 0 
\end{equation}
the local martingale \eqref{eq:locmartingone} is a true martingale. Taking expectations and eliminating $\epsilon$ with dominated convergence yields the above formula.
\end{proof}

Theorem \ref{thm:nonlocformone} is precisely \cite[Theorem~5.2]{ElworthyLi}, which was also obtained in \cite{LiThompson} by a slightly different method.

\begin{theorem}\label{thm:nonlocformhalf}
Suppose $x_0 \in M$, $f \in \mathcal{B}_b$, $V$ bounded below and $T>0$. Set
\begin{equation}
k_s = \frac{T-2s}{T}\vee 0,\quad l_s = 1 \wedge \frac{2(T-s)}{T}.
\end{equation}
Suppose $|{\Ric}_Z|$, $|dV|$ and $|Z|$ are bounded with $V$ bounded below. Then
\begin{equation}
\begin{split}
&L(P^V_Tf)(x_0) \\
=\text{ }&\E\left[ \V_T f(x_T)\int_0^T \langle \dot{k}_s Z,//_s dB_s\rangle \right]\\
& +\frac{1}{2}\E\left[ \V_T f(x_T)\left( \int_0^T \<  W_s \dot{l}_s ,//_s dB_s\> + dV_{T-s}(W_s l_s)ds\right)\int_0^T \dot{k}_s W_s^{-1} //_s dB_s\right].
\end{split}
\end{equation}
\end{theorem}

\begin{proof}
It follows from Corollary \ref{cor:globesthalf} that $|L P^V_t|$ is bounded on $[\epsilon,T] \times M$. Therefore, using $k_s$ and
\begin{equation}
l^\epsilon_s = \left(1 \wedge \frac{T-\epsilon-s}{\frac{T}{2}-\epsilon}\right)\vee 0
\end{equation}
the local martingale appearing in Lemma \ref{lem:locmartonehalf} is a true martingale. Taking expectations, using Lemma \ref{lem:potlemma} as in the proof of Theorem \ref{thm:formhalf} and eliminating $\epsilon$ with dominated convergence yields the above formula.
\end{proof}

\begin{theorem}\label{thm:nonlocformtwo}
Suppose $x_0\in M$ with $v,w \in T_{x_0}M$, $f \in \mathcal{B}_b$, $V$ bounded below and $T>0$. Define $k_s$ and $l_s$ as in Theorem \ref{thm:nonlocformhalf}. Suppose $|{\Ric}_Z|$, $|dV|$, $|\nabla d V|$, $|\nabla{\Ric}^\sharp_Z+d^\star R-2R(Z)|$ and $|R|$ are bounded with $V$ bounded below. Then
\begin{equation}
\begin{split}
&(\nabla d P^V_Tf)(v,w) \\
=\text{ }&  - \E \left[ \V_T f(x_T)\int_0^T\langle  W^{\prime}_s(\dot{k}_sv,w), //_s dB_s\rangle\right]\\
&-\E\left[ \V_T f(x_T)\int_0^T ((\nabla d V_{T-s})(W_s(k_sv),W_s(w))+(dV_{T-s})(W^{\prime}_s(k_sv,w)))ds\right]\\
&+\E\left[ \V_T f(x_T)\left(\int_{0}^T \< W_s( \dot{l}_sw),//_s dB_s\> + dV_{T-s}( W_s(l_sw)) ds\right)\right.\\
&\quad\quad\quad\quad\quad\quad\quad\quad\quad\quad\quad\quad\quad\quad\quad\left.\cdot\left(\int_0^{T} \<W_s (\dot{k}_sv),//_sdB_s\>+ dV_{T-s}(W_s(k_sv))ds\right)\right].
\end{split}
\end{equation}
\end{theorem}

\begin{proof}
It follows from Corollary \ref{cor:globesttwo} that $|d P^V_t|$ and $|\nabla d P^V_t|$ are bounded on $[\epsilon,T] \times M$. Therefore, using $l^\epsilon_s$ defined as in the proof of Theorem \ref{thm:nonlocformhalf}, the local martingale appearing in Lemma \ref{lem:locmarttwo} is a true martingale. Taking expectations, proceeding as in the proof of Theorem \ref{thm:locformtwo} and eliminating $\epsilon$ with dominated convergence yields the above formula.
\end{proof}

For the case $V=0$, Theorem \ref{thm:nonlocformtwo} gives the filtered version of the second part of \cite[Theorem~3.1]{ElworthyLi}, which was proved by differentiating under the expectation for $f \in BC^2$ and which, as observed in \cite{Plank}, contains a slight error, permuting the vectors $v$ and $w$.

\begin{remark}
Our gradient and Hessian formulae require $V\in C^1$ and $V \in C^2$, respectively (see Remark \ref{rem:notsmooth}). More generally, it is desirable to consider possibly very singular potentials, such as those which appear in many quantum mechanical problems. See, for example, \cite{Guneysu} and \cite{PriolaWang}. It was pointed out to the authors of \cite{ElworthyLi} by G. Da Prato, and to the author of this article by X.-M. Li, that non-smooth potentials $V$ can be dealt with using the variation of constants formula:
\begin{equation}\label{eq:vocform}
P^V_Tf = P_Tf - \int_0^T P_{T-s}(V_s P^V_s f)ds
\end{equation}
where $P_T$ denotes the minimal semigroup associated to the operator $L$. So long as $P^V_Tf$ is sufficiently regular, formulae and estimates $dP^V_Tf$ can be obtained from formulae and estimates for $dP_Tf$, simply by differentiating the above formula. In particular, this approach results in gradient estimates depending only on $\| V\|_\infty$ (like those in \cite{PriolaWang} for domains in $\mathbb{R}^n$). Our gradient estimate, Theorem \ref{thm:locestone}, on the other hand, does not require that $V$ is bounded (only bounded below). For the second derivatives one must take care in passing the derivatives through the integral in formula \eqref{eq:vocform}. For the case in which the potential is a bounded H\"{o}lder continuous function $V$ which does not depend on time, this can be achieved at each point $x_0 \in M$ by shifting $V$ to $V(x_0)=0$. The details of this for, the Hessian, are given in \cite{Lihess}, where the approach taken in based on that of \cite{ElworthyLi}.
\end{remark}

\section{Kernel Estimates}\label{sec:kerests}

Now suppose $Z = \nabla h$, for some $h \in C^2$, and consider the $m$-dimensional Bakry-Emery curvature tensor
\begin{equation}
{\Ric}_{m,n} := {\Ric}^\sharp - \nabla d h - \frac{\nabla h \otimes \nabla h}{m-n}
\end{equation}
where $m \geq n$ is a constant (see \cite{Lott}). Denoting by $p^h_t(x,y)$ the density of the diffusion with generator $L$, with respect to the weighted Riemannian measure $e^{h}dy$, it follows, as explained in the proof of \cite[Theorem~1.4]{XDLi}, that if $\Ric_{m,n} \geq -\kappa$ for some $\kappa \geq 0$ then there exists a positive constant $C\equiv C(\kappa,m,T)$ such that
\begin{equation}
\log \left( \frac{p^h_{\frac{t}{2}}(x,z)}{p^h_{t}(x,y)}\right) \leq C\left( 1+ \frac{d^2(x,y)}{t}\right)
\end{equation}
for all $x,y,z \in M$ and $t \in (0,T]$. Assuming $V$ is bounded, it follows that the same inequality holds for the integral kernel $p^{h,V}_t(x,y)$ of the semigroup $P^V_tf$, since
\begin{equation}
p^{h,V}_t(x,y) = p^{h}_t(x,y)\E\left[ \V_t | x_0=x, x_t =y\right]
\end{equation}
by the Feynman-Kac formula. We can therefore derive from Theorems \ref{thm:nonlocformone}, \ref{thm:nonlocformhalf} and \ref{thm:nonlocformtwo} estimates on the logarithmic derivatives of $p^{h,V}_t(x,y)$ by using Jensen's inequality (as in \cite[Lemma~6.45]{Stroock2000}). In particular, the assumptions of Theorem \ref{thm:nonlocformone} with $Z=\nabla h$ plus boundedness of $V$ and a lower bound on ${\Ric}_{m,n}$ imply the existence of a constant $C_1(T)$ such that
\begin{equation}
|d \log p^{h,V}_t(\cdot,y)_x|^2 \leq C_1(T)\left(\frac{1}{t}+\frac{d^2(x,y)}{t^2}\right)
\end{equation}
for all $x,y \in M$ and $t \in (0,T]$. The details of this (for the case $h=0$) can be found in \cite{LiThompson}. Similarly, the assumptions of Theorem \ref{thm:nonlocformhalf} with $Z=\nabla h$ plus boundedness of $V$ and a lower bound on ${\Ric}_{m,n}$ imply for the Witten Laplacian $\Delta_h := \frac{1}{2}\Delta + \nabla h$ the existence of a constant $C_2(T)$ such that
\begin{equation}
|\Delta_h \log p^{h,V}_t(\cdot,y)(x)| \leq C_2(T)\left(\frac{1}{t}+\frac{d^2(x,y)}{t^2}\right)
\end{equation}
for all $x,y \in M$ and $t \in (0,T]$. Finally, the assumptions of Theorem \ref{thm:nonlocformtwo} with $Z=\nabla h$ plus boundedness of $V$ and a lower bound on ${\Ric}_{m,n}$ imply the existence of a constant $C_3(T)$ such that
\begin{equation}
|\nabla d \log p^{h,V}_t(\cdot,y)_x| \leq C_3(T)\left(\frac{1}{t}+\frac{d^2(x,y)}{t^2}\right)
\end{equation}
for all $x,y \in M$ and $t \in (0,T]$.

\end{document}